\documentclass{amsart}
\usepackage{graphicx}
\usepackage{amsmath}
\usepackage{amsthm}
\usepackage{amssymb,bbm}%
\usepackage[numbers, square]{natbib}

\newcommand{\E}{\mathbb E}

\newcommand{\R}{\mathbb{R}}
\newcommand{\N}{\mathbb{N}}

\newcommand{\Z}{\mathbb{Z}}

\renewcommand{\P}{\mathbb{P}}

\newcommand{\relint}{\mathop{\mathrm{relint}}}
\newcommand{\interior}{\mathop{\mathrm{int}}}

\newcommand{\conv}{\mathop{\mathrm{conv}}\nolimits}
\newcommand{\aff}{\mathop{\mathrm{aff}}\nolimits}

\newcommand{\cone}{\mathop{\mathrm{cone}}\nolimits}

\newcommand{\dd}{{\rm d}}
\newcommand{\cF}{\mathcal{F}}

\newcommand{\eps}{\varepsilon}
\newcommand{\eqdistr}{\stackrel{d}{=}}

\newcommand{\ind}{\mathbbm{1}}

\theoremstyle{plain}
\newtheorem{theorem}{Theorem}[section]
\newtheorem{lemma}[theorem]{Lemma}

\theoremstyle{definition}

\newtheorem{example}[theorem]{Example}

\theoremstyle{remark}
\newtheorem{remark}[theorem]{Remark}

\begin{document}

\title
[Inclusion-exclusion principles]
{Inclusion-Exclusion principles for convex hulls and the Euler relation}

\author{Zakhar Kabluchko}

\address{Zakhar Kabluchko, Universit\"at M\"unster, Institut f\"ur Mathematische Statistik, Orl\'{e}ans-Ring 10, 48149 M\"unster, Germany} \email{zakhar.kabluchko@uni-muenster.de}

\author{G\"unter Last}
\address{G\"unter Last, Karlsruher Institut f\"ur Technologie, Englerstr.\ 2, D-76131 Karlsruhe, Germany}
\email{guenter.last@kit.edu }

\author{Dmitry Zaporozhets}
\address{Dmitry Zaporozhets\\
St.\ Petersburg Department of
Steklov Institute of Mathematics,
Fontanka~27,
 191011 St.\ Petersburg,
Russia}
\email{zap1979@gmail.com}

\begin{abstract}
Consider $n$ points $X_1,\ldots,X_n$ in $\R^d$ and denote their convex hull by $\Pi$.
We prove a number of inclusion-exclusion identities for the system of convex hulls $\Pi_I:=\conv(X_i\colon i\in I)$, where $I$ ranges over all subsets of $\{1,\ldots,n\}$.
For instance, denoting by $c_k(X)$ the number of $k$-element subcollections of $(X_1,\ldots,X_n)$ whose convex hull contains a point $X\in\R^d$, we prove that
$$
c_1(X)-c_2(X)+c_3(X)-\ldots + (-1)^{n-1} c_n(X) = (-1)^{\dim \Pi}
$$
for \textit{all} $X$ in the relative interior of $\Pi$. This confirms a conjecture of R.\ Cowan [Adv.\ Appl.\ Probab.,\ 39(3):630--644, 2007] who proved the above formula  for \textit{almost all} $X$.
We establish similar results for the number of polytopes $\Pi_J$ containing a given polytope $\Pi_I$ as an $r$-dimensional face, thus proving another conjecture of R.\ Cowan [Discrete Comput.\ Geom.,\ 43(2):209--220, 2010]. As a consequence, we derive inclusion-exclusion identities for the intrinsic volumes and the face numbers of the polytopes $\Pi_I$.
The main tool in our proofs is a formula for the alternating sum of the face numbers of a convex polytope intersected by an affine subspace. This formula generalizes the classical Euler--Schl\"afli--Poincar\'e relation and is of independent interest.
\end{abstract}


\maketitle

\section{Statement of results}
\subsection{Introduction}
Let $X_1,\ldots,X_n$ be a finite collection\footnote{Since we do not require the points to be distinct, we use the notions ``collection'', ``subcollection'', etc.\ rather than ``set'', ``subset'', etc.} of points in $\R^d$.
Denote its convex hull by $\Pi$:
$$
\Pi:= \conv(X_1,\ldots,X_n) = \left\{\sum_{i=1}^n \lambda_i X_i \colon \lambda_1,\ldots,\lambda_n\geq 0,  \sum_{i=1}^n \lambda_i = 1\right\}.
$$
More generally, for any set $I\subset\{1,\ldots,n\}$ we write
$$
\Pi_I:=\conv(X_i\colon i\in I) 
$$
for the convex hull of the points $X_i, i\in I$. In this paper, we are interested in various inclusion-exclusion relations satisfied by the system of polytopes $\Pi_I$, where $I$ ranges over all subsets of  $\{1,\ldots,n\}$.

Recall that a \textit{polytope} is a convex hull of a finite set of points. By definition, it is convex and compact. The \textit{interior} of a polytope $P$ is denoted by $\interior P$.  The \textit{relative interior} of a polytope $P$, denoted by $\relint  P$, is the interior of $P$ with respect to its affine hull. For example, the relative interior of a point is this point itself. For these and other standard definitions from convex geometry we refer to the monographs by Schneider~\cite{schneider_book}, Schneider and Weil~\cite{schneider_weil_book}, and Gr\"unbaum~\cite{gruenbaum_book}.
We denote by $\cF_k(P)$ the set of $k$-dimensional (closed) faces of $P$, and by $f_k(P)$ their number. Let $\dim P$ be the dimension of $P$ and write $\cF_\bullet(P)=\cup_{k=0}^{\dim P} \cF_k(P)$ for the set of all faces of $P$ including $P$ itself. Of central importance for the present paper is the classical Euler--Schl\"afli--Poincar\'e relation (see~\cite[page~626]{schneider_weil_book} or~\cite[page~130]{gruenbaum_book}) which states that for every polytope $P$,
\begin{equation}\label{eq:euler_relation}
\sum_{F\in \cF_{\bullet} (P)} (-1)^{\dim F} = \sum_{k=0}^{\dim P} (-1)^k f_k(P) = 1.
\end{equation}

\subsection{Cowan's formula}
Recall that $X_1,\ldots,X_n$ is a finite collection of points in $\R^d$, and additionally let $X$ be any point in $\R^d$.  For $k=1,\ldots,n$ denote by $c_k(X)$ the number of $k$-element subcollections of $(X_1,\ldots,X_n)$ containing $X$ in their convex hull:
\begin{equation}\label{eq:def_c_k}
c_k(X) = \# \{I\subset\{1,\ldots,n\}\colon \# I = k, X\in \Pi_I\}.
\end{equation}
Here, $\#B$ is the number of elements in a set $B$.
Cowan~\cite{cowan1}  proved that 
\begin{equation}\label{eq:cowan_exceptional}
\sum_{k=1}^{n} (-1)^{k-1} c_k(X)
=
\begin{cases}
(-1)^{\dim \Pi}, &\text{if } X \in  (\relint \Pi)\backslash \mathcal E,\\
0, &\text{if } X\notin \relint  \Pi,
\end{cases}
\end{equation}
where  $\mathcal E$ is some ``exceptional set'' of codimension $2$. Cowan also conjectured that, in fact, the first case of the formula holds for \emph{all} $X\in \relint  \Pi$ and proved this for $d=2$. Another proof of Cowan's formula can be found in the book of Schneider and Weil~\cite[p.~309--310]{schneider_weil_book}, but there is again an exceptional set, namely the union of all $(d-2)$-dimensional affine subspaces spanned by $X_1,\ldots,X_n$. Our first result confirms Cowan's conjecture for all $d\in\N$.

\begin{theorem}\label{theo:cowan}
For any finite collection of points $X_1,\ldots,X_n$ in $\R^d$ and for all $X\in \R^d$ we have
\begin{equation}\label{eq:cowan}
\sum_{k=1}^{n} (-1)^{k-1} c_k(X)
=
\begin{cases}
(-1)^{\dim \Pi}, &\text{if } X \in  \relint \Pi,\\
0, &\text{if } X\notin \relint  \Pi.
\end{cases}
\end{equation}
\end{theorem}
This formula should be compared to the well-known inclusion-exclusion principles which state that for arbitrary subsets $A_1,\ldots,A_n$ of a set $\Omega$,
\begin{align}
&\sum_{\varnothing\neq I\subset\{1,\ldots,n\}} (-1)^{\# I-1} \ind_{\cap_{i\in I} A_i} = \ind_{A_1\cup\ldots\cup A_n},\label{eq:inclusion1}\\
&\sum_{\varnothing\neq I\subset\{1,\ldots,n\}} (-1)^{\# I-1} \ind_{\cup_{i\in I} A_i} = \ind_{A_1\cap\ldots\cap A_n},\label{eq:inclusion2}
\end{align}
where $\ind_B$ denotes the indicator function of a set $B$.
In terms of indicator functions, Cowan's formula can be written as
\begin{equation}\label{eq:cowan_indicator}
\sum_{\varnothing\neq I\subset\{1,\ldots,n\}} (-1)^{\# I-1} \ind_{\Pi_I} = (-1)^{\dim \Pi} \ind_{\relint  \Pi}.
\end{equation}
This is clearly analogous to~\eqref{eq:inclusion2} if we consider the convex hull as an analogue of the union, and the interior of the convex hull with a ``phase factor'' $(-1)^{\dim \Pi}$ as an analogue of the intersection.  The next theorem states a ``dual'' Cowan's formula which is analogous  to~\eqref{eq:inclusion1}.
\begin{theorem}\label{theo:cowan_dual}
For any finite collection of points $X_1,\ldots,X_n$ in $\R^d$ we have
\begin{equation}\label{eq:cowan_dual}
\sum_{\varnothing\neq I\subset\{1,\ldots,n\}} (-1)^{\# I -1}  (-1)^{\dim \Pi_I}\ind_{\relint  \Pi_I}  = \ind_{\Pi}.
\end{equation}
\end{theorem}
Given~\eqref{eq:cowan}, we can obtain~\eqref{eq:cowan_dual} by simple algebraic manipulations (M\"obius inversion), see Section~\ref{subsec:proof_dual_cowan}. The proof of~\eqref{eq:cowan} is non-trivial and will be given in Section~\ref{subsec:proof_cowan}.

\subsection{Generalized Euler relation}
The left-hand side of~\eqref{eq:cowan} looks very much like the Euler characteristic, but it seems that the original proof of Cowan~\cite{cowan1} does not establish any direct connection between his formula and the theory of additive functionals.
We will follow a different method. Referring to Section~\ref{subsec:proof_cowan} for more details, we briefly describe the essence of our approach.
Consider a non-degenerate simplex with $n$ vertices located in some $(n-1)$-dimensional linear space $H$. Define an affine map $A:H\to\R^d$ by sending the vertices of the simplex to the points $X_1,\ldots,X_n$. Then, the polytopes $\Pi_I$, $I\subset\{1,\ldots,n\}$, are the images of the faces of the simplex. Passing to the preimages, we can interpret Cowan's formula as a statement about the intersections between the affine subspace $A^{-1}(X)$ and the faces of the simplex.

The following general fact (which may be of independent interest) is the main ingredient of our proofs. Although it may be known, we were unable to find it in the literature.

\begin{theorem}\label{theo:euler_intersection}
Let $T$ be a polytope in $\R^m$ with non-empty interior $\interior T$. Let $L\subset\R^m$ be an affine subspace of dimension $m-d$.
Denote by $a_k$ the number of $k$-dimensional faces of $T$ which are intersected by $L$, where $k=0,\ldots, m$. Then,
\begin{equation}\label{eq:euler_intersection}
\sum_{k=0}^{m} (-1)^k a_k
=
\begin{cases}
(-1)^{d}, &\text{if } L\cap \interior  T\neq \varnothing,\\
0, &\text{if } L\cap \interior  T = \varnothing.
\end{cases}
\end{equation}
\end{theorem}
In the special case $L=\R^m$ we have $a_k=f_k(T)$,  the number of $k$-dimensional faces of $T$, and the theorem reduces to the classical Euler relation~\eqref{eq:euler_relation}.
\begin{example}
Consider a square $ABCD$ and a line $L$ passing through $A$ and the middle of the side $BC$. Then, $a_0=1$ (vertex $A$), $a_1=3$ (sides $AB$, $AD$, $BC$), and $a_2=1$. We have $1-3+1=-1$. Let $L'$ be the line passing through $A$ and $B$. Then, $a_0'=2$ (vertices $A$ and $B$), $a_1'=3$ (sides $AB$, $AD$, $BC$), and $a_3'=1$. We have $2-3+1=0$.
\end{example}

If for every face $F$ of the polytope $L\cap T$ there is a \textit{unique} face $G$ of the polytope $T$ such that $G\cap L=F$, and if $\dim F = \dim G - d$, then~\eqref{eq:euler_intersection} is a consequence of the Euler--Schl\"afli--Poincar\'e relation~\eqref{eq:euler_relation} for the polytope $L\cap T$. However, it is easy to construct examples in which the uniqueness fails (for example, if $L$ is a line containing some vertex of $T$, $m\geq 2$). Thus, the main problem is how to treat these ``non-general position'' cases.

Our proof of Theorem~\ref{theo:euler_intersection} (which will be given in Section~\ref{sec:proof_euler}) is based on Groemer's extension of the Euler characteristic to the class of ro-polytopes. Our method can be applied to obtain further results of the same type, for example the following one. Recall that $\cF_k(P)$ is the set of $k$-dimensional faces of a polytope  $P$. 

\begin{theorem}\label{theo:euler_touch}
Let $T_1,T_2\subset\R^m$ be two polytopes which touch each other, that is $(\relint T_1)\cap (\relint T_2) = \varnothing$ but $T_1\cap T_2 \neq\varnothing$. Then,
\begin{equation}\label{eq:euler_touch}
\sum_{k=0}^{m} (-1)^k \#\{F\in \cF_{k}(T_1)\colon F\cap T_2\neq \varnothing\} = 0.
\end{equation}
\end{theorem}

A somewhat related result is the Euler relation for face-to-face tessellations, see~\cite[Eq.~(14.65)]{schneider_weil_book}, but it seems that this relation implies neither Theorem~\ref{theo:euler_intersection} nor Theorem~\ref{theo:euler_touch}.

\subsection{Inclusion-exclusion principles for intrinsic volumes}
It is possible to generalize Cowan's formula to the setting when we count convex hulls intersecting certain affine subspace $F\subset\R^d$ rather than convex hulls containing a given point $X$. Let $X_1,\ldots,X_n$ be a finite collection of points in  $\R^d$. It will be convenient to assume that the convex hull $\Pi$ of these points has full dimension $d$. This is not a restriction of generality because otherwise we could replace $\R^d$ by the affine hull of $X_1,\ldots,X_n$.   Given an affine subspace $F\subset\R^d$ and $k=1,\ldots,n$, denote by $c_k(F)$ the number of $k$-element subcollections of $(X_1,\ldots,X_n)$ whose convex hull intersects $F$, that is
$$
c_k(F) = \# \{I\subset\{1,\ldots,n\}\colon \# I = k, F\cap \Pi_I \neq \varnothing\}.
$$
\begin{theorem}\label{theo:cowan_affine}
Under the above assumptions,
\begin{equation}\label{eq:cowan_affine}
\sum_{k=1}^{n} (-1)^{k-1} c_k(F)
=
\begin{cases}
(-1)^{d-\dim F}, &\text{if } F\cap \interior  \Pi \neq \varnothing,\\
0, &\text{if } F\cap \interior  \Pi = \varnothing.
\end{cases}
\end{equation}
\end{theorem}
The above result reduces to Cowan's formula if $F$ is a point. In terms of indicator variables, Theorem~\ref{theo:cowan_affine} can be written as follows:
\begin{equation}\label{eq:cowan_affine_indicator}
\sum_{\varnothing\neq I\subset\{1,\ldots,n\}} (-1)^{\# I-1} \ind_{\{F \cap \Pi_I\neq \varnothing\}} = (-1)^{d-\dim F} \ind_{\{F\cap \interior  \Pi\neq \varnothing\}}.
\end{equation}

The set of all $r$-dimensional affine subspaces of $\R^d$ is denoted by $\text{AffGr}(d,r)$ and called the affine Grassmannian. It is known that $\text{AffGr}(d,r)$ carries a measure $\mu_r$ (defined up to a multiplicative constant) invariant with respect to the natural action of the isometry group of $\R^d$; see~\cite[Chapter~5.1]{schneider_weil_book}.   The intrinsic volumes $V_0(K),\ldots,V_d(K)$ of a compact convex set $K\subset \R^d$ satisfy the Crofton formula~\cite[Theorem~5.1.1]{schneider_weil_book}
\begin{equation}\label{eq:crofton}
V_{r}(K) = 
\frac{\Gamma\left(\frac{d-r+1}{2}\right)\Gamma\left(\frac{r+1}{2}\right)}{\Gamma\left(\frac 12\right) \Gamma\left(\frac{d+1}{2}\right)}
\int_{\text{AffGr}(d,d-r)} \ind_{\{F\cap K\neq \varnothing\}} \mu_{d-r}(\dd F), \; r=0,\ldots,d,
\end{equation}
where we used the same normalization for $\mu_r$ as in~\cite[Chapter~5.1]{schneider_weil_book}.
Integrating~\eqref{eq:cowan_affine_indicator} over $\text{AffGr}(d,d-r)$ with respect to  $\mu_{d-r}(\dd F)$, we obtain an inclusion-exclusion principle for intrinsic volumes which generalizes the result of Cowan~\cite{cowan1} who considered the case $r=d$.
\begin{theorem}\label{theo:cowan_intrinsic}
For any finite collection $X_1,\ldots,X_n$ of points in $\R^d$ with $\dim \Pi=d$ and for every $r=0,\ldots,d$, we have
\begin{equation}\label{eq:cowan_intrinsic}
\sum_{\varnothing\neq I\subset\{1,\ldots,n\}} (-1)^{\# I-1} V_{r}(\Pi_I) = (-1)^{r} V_{r}(\Pi).
\end{equation}
\end{theorem}
\begin{remark}
Using the local Crofton formula~\cite[Theorem~5.3.3]{schneider_weil_book} one can prove a similar identity for the curvature measures of the $\Pi_I$'s.
\end{remark}

As in the work of Cowan~\cite{cowan1}, it is possible to obtain probabilistic corollaries of the above deterministic results.
Let  $X_1,\ldots,X_n$ be random vectors with values in $\R^d$. We call $(X_1,\ldots,X_n)$ an \textit{exchangeable tuple} if for every permutation $\sigma$ of $\{1,\ldots,n\}$, the distributional equality
\begin{equation}\label{eq:exchange}
(X_1,\ldots,X_n) \eqdistr (X_{\sigma(1)},\ldots, X_{\sigma(n)})
\end{equation}
holds. For example, this condition is satisfied if $X_1,\ldots,X_n$ are independent identically distributed random vectors; see~\cite{cowan1} for more examples.  Given any subset  $I\subset\{1,\ldots,n\}$ with $\#I = k$  write
$$
v_{r}(k)
:=
\E V_r(\Pi_I)
=
\E V_r(\conv(X_1,\ldots,X_k)),
$$
where we stress that  by exchangeability,  there is no dependence on the choice of $I$.

\begin{theorem}\label{theo:cowan_intrinsic_random}
Let $(X_1,\ldots,X_n)$ be an exchangeable tuple of random vectors in $\R^d$ such that $\dim \Pi=d$ a.s. Let $r\in \{0,\ldots,d\}$ and assume that $\E V_r( \Pi)<\infty$. Then,
$$
\sum_{k=1}^{n} (-1)^{k-1} \binom{n}{k} v_{r}(k) = (-1)^{r} v_{r}(n).
$$
\end{theorem}
\begin{remark}
Condition $\E V_r(\Pi)<\infty$ holds provided that  $\E |X_1|^r<\infty$.  This follows
from the fact that
$\Pi$ is contained in the ball of radius $|X_1|+\ldots+|X_n|$ centred at the origin together with
the monotonicity (see~\eqref{eq:crofton}) and the homogeneity of the intrinsic volumes.
\end{remark}
\begin{proof}[Proof of Theorem~\ref{theo:cowan_intrinsic_random}]
By the monotonicity of the intrinsic volumes, our assumption
 $\E V_r(\Pi)<\infty$ implies that $v_{r}(k) <\infty$ for all $k\in \{1,\ldots,n\}$. Taking the expectation in~\eqref{eq:cowan_intrinsic} and noting that there are $\binom nk$ subsets $I$ with $k$ elements, we obtain the result.
\end{proof}
In the case $r=d$, Theorem~\ref{theo:cowan_intrinsic_random} reduces to the well-known identity of Buchta~\cite{buchta}; see also~\cite[Theorem~8.2.6]{schneider_weil_book} and~\cite{affentranger1,affentranger2,badertscher,cowan1,beermann_reitzner,buchta_miles}.

\subsection{Inclusion-exclusion principles for faces}
Next we are going to state inclusion-exclusion principles for the faces of the polytopes $\Pi_I$, $I\subset \{1,\ldots,n\}$. These deterministic formulas will be used to prove probabilistic identities (conjectured by Cowan in~\cite{cowan2}) on the expected face numbers of random convex hulls.

Fix some finite collection of points $X_1,\ldots,X_n$ in $\R^d$.
For a set $I\subset\{1,\ldots,n\}$
denote by $b_j(I)$ the number of $j$-element sets $J\supset  I$ such that $\Pi_I$ is a face of $\Pi_J$, that is
$$
b_j(I) = \#\{J\subset\{1,\ldots,n\}\colon J\supset  I, \#J=j, \Pi_I \text{ is a face of } \Pi_J\}.
$$
\begin{theorem}\label{theo:inclusion_exclusion_faces}
Consider a non-empty subset $I\subset\{1,\ldots,n\}$ such that\footnote{We write $(X_k \colon k\notin I)$ rather than $\{X_k \colon k\notin I\}$ because the points need not be distinct.} $\Pi_I\cap (X_k \colon k\notin I)=\varnothing$. Then,
\begin{equation}\label{eq:cowan_faces}
\sum_{j=\#I}^n (-1)^{j-1} b_j(I)=
\begin{cases}
(-1)^{\dim \Pi + \#I - 1 - \dim \Pi_I}, &\text{if } \Pi_I \text{ is not a face of } \Pi,\\
0, &\text{if } \Pi_I \text{ is a face of } \Pi.
\end{cases}
\end{equation}
\end{theorem}

The above assumption on $I$ is always satisfied if $\#I=1$ (that is, if $\Pi_I$ is a single point) and the points $X_1,\ldots,X_n$ are distinct. The following example shows that~\eqref{eq:cowan_faces} may fail in general.

\begin{example}\label{ex:example}
Consider the points $X_i=i$, $i=1,\ldots,5$, on the real line and take $I=\{2,4\}$. Clearly, $\Pi_I= [2,4]$ is a face of $\Pi_J$ if and only if $J=\{2,4\}$ or  $J=\{2,3,4\}$. Thus, $b_1(I)=b_4(I)=b_5(I)=0$ and $b_2(I)=b_3(I)=1$. The alternating sum in~\eqref{eq:cowan_faces} equals $0$, so that~\eqref{eq:cowan_faces} fails.
\end{example}

In order to state a version of Theorem~\ref{theo:inclusion_exclusion_faces} not requiring additional assumptions on $I$, we have to introduce the following modified version of $b_j(I)$:
$$
b_j^*(I) = \#\{J\subset\{1,\ldots,n\}\colon J\supset  I, \#J=j, \Pi_I \text{ is a clean face of } \Pi_J\},
$$
where the word ``clean'' means that $\Pi_I\cap (X_k\colon k\in J\backslash I)=\varnothing$.

\begin{theorem}\label{theo:inclusion_exclusion_faces_clean}
For every non-empty subset $I\subset\{1,\ldots,n\}$ we have
\begin{equation}\label{eq:cowan_faces_clean}
\sum_{j=\#I}^n (-1)^{j-1} b_j^*(I)=
\begin{cases}
(-1)^{\dim \Pi + \#I - 1 - \dim \Pi_I}, &\text{if } \Pi_I \text{ is not a face of } \Pi,\\
0, &\text{if } \Pi_I \text{ is a face of } \Pi.
\end{cases}
\end{equation}
\end{theorem}
Under the assumption $\Pi_I\cap (X_k\colon k\notin I)=\varnothing$ the conditions ``$\Pi_I$ is a face of $\Pi_J$'' and ``$\Pi_I$ is a clean face of $\Pi_J$'' become equivalent which means that  Theorem~\ref{theo:inclusion_exclusion_faces_clean} contains Theorem~\ref{theo:inclusion_exclusion_faces} as a special case.
\begin{example}
Continuing Example~\ref{ex:example}, we see that $\Pi_I$ is a clean face of $\Pi_J$ if and only if $J=I=\{2,4\}$. Hence, all $b_j^*(I)$ equal zero except for $b_2^*(I)=1$, and~\eqref{eq:cowan_faces_clean} holds.
\end{example}


We say that a finite collection $X_1,\ldots,X_n$ of points in $\R^d$ is in \textit{$r$-general position} for some $r\in \{1,\ldots,d\}$ if every $r$-dimensional affine subspace contains at most $r+1$ points from this set.
Recall that $f_r(P)=\#\cF_r(P)$ denotes the number of $r$-dimensional faces of a polytope $P$. The next result is a deterministic counterpart of a probabilistic formula conjectured by Cowan~\cite{cowan2}. We obtain it by summing up~\eqref{eq:cowan_faces} over all subsets $I$ with $\# I = r+1$; see Section~\ref{subsec:proof_theo_relation_for_f_r} for details.
\begin{theorem}\label{theo:relation_for_f_r}
Fix some $d\in\N$ and  $r\in \{1,\ldots,d\}$. Let $X_1,\ldots,X_n$ be a finite collection of points in $r$-general position in $\R^d$  and suppose that $\dim \Pi=d$. Then,
\begin{equation}\label{eq:relation_for_f_r}
\sum_{\varnothing\neq J\subset\{1,\ldots,n\}} (-1)^{\#J-1} f_r(\Pi_J) = (-1)^d \left(\binom {n}{r+1} - f_r(\Pi)\right).
\end{equation}
\end{theorem}
In particular, if $n-d$ is odd, then the term $f_r(\Pi)$ appears on both sides with different signs, and we obtain the relation
$$
2 f_r(\Pi) = \binom{n}{r+1} + (-1)^n \sum_{\varnothing \neq J\subsetneq \{1,\ldots,n\}} (-1)^{\#J-1} f_r(\Pi_J),
$$
where we stress that the term with $J=\{1,\ldots,n\}$ is excluded from the summation. If $n-d$ is even, then the term $f_r(\Pi)$ cancels and we obtain
$$
\sum_{\varnothing \neq J\subsetneq \{1,\ldots,n\}} (-1)^{\#J-1} f_r(\Pi_J) = (-1)^d \binom{n}{r+1}.
$$
Passing to the random setting, we prove a formula which was conjectured by Cowan~\cite{cowan2} and proved by him in some special cases using the Dehn--Sommerville relations. Let $(X_1,\ldots,X_n)$ be an exchangeable tuple of random vectors in $\R^d$; see~\eqref{eq:exchange}. For an arbitrary subset $I\subset\{1,\ldots,n\}$  with $\# I=k$ we write
$$
F_r(k) := \E f_r(\Pi_I)
=
\E f_r(\conv(X_1,\ldots,X_k))
$$
for the expected number of $r$-dimensional faces of $\Pi_I$.
\begin{theorem}\label{theo:relation_for_f_r_random}
Fix some $d\in\N$ and $r\in \{1,\ldots,d\}$. Let $(X_1,\ldots,X_n)$ be an exchangeable tuple of random vectors in $\R^d$ such that with probability one, $\dim \Pi=d$ and the points $X_1,\ldots,X_n$ are in $r$-general position. Then,
$$
\sum_{k=1}^n (-1)^{k-1}  \binom nk F_r(k) = (-1)^d \left(\binom {n}{r+1} - F_r(n)\right).
$$
\end{theorem}
\begin{proof}
Take the expectation in~\eqref{eq:relation_for_f_r} and note that there are $\binom nk$ subsets $J$ with $k$ elements.
\end{proof}


\subsection{A proof of Buchta's identity}
Several remarkable identities for random convex hulls were discovered by Buchta in~\cite{buchta_identity}. One of these identities has a form very similar to the inclusion-exclusion principles studied in the present paper. Given that Theorems~\ref{theo:cowan_intrinsic_random} and~\ref{theo:relation_for_f_r_random} have deterministic counterparts, it is natural to ask whether something similar is true for Buchta's identity.

To state Buchta's identity, let $X_1,\ldots,X_n$ be independent identically distributed random vectors in $\R^d$. The  probability distribution of $X_i$ is denoted by $\mu$ and assumed to be non-atomic (which implies that $X_1,\ldots,X_n$ are distinct a.s.).
Denote by $N_n=f_0(\Pi)$ the number of vertices of $\Pi$ and write
$$
M_j := \mu (\conv(X_1,\ldots,X_j))
$$
for the so-called probability content of $\conv(X_1,\ldots,X_j)$.  Buchta's identity~\cite{buchta_identity} (see also~\cite[Theorem 8.2.5]{schneider_weil_book}) states that for every $l=1,\ldots,n$,
\begin{equation}\label{eq:buchta}
\P[N_n = l] = (-1)^l \binom n l \sum_{j=1}^l (-1)^j \binom lj \E M_j^{n-j}.
\end{equation}
If $\mu$ is the uniform distribution on some convex body $K\subset \R^d$, then $M_j$ is just the volume of $\conv(X_1,\ldots,X_j)$ divided by the volume of $K$. Below we provide a ``pointwise'' version of~\eqref{eq:buchta} which turns out to be a very simple  inclusion-exclusion formula. Our proof is different from the original proof of Buchta~\cite{buchta} (see also~\cite[Theorem 8.2.5]{schneider_weil_book}).

\vspace*{2mm}
\noindent
\textit{Proof of~\eqref{eq:buchta}.}
Denote the probability space on which $X_1,\ldots,X_n$ are defined by $(\Omega,\mathbb F, \P)$. Let $A_1,\ldots,A_n\in \mathbb F$ be  random events to be specified later and denote by $A_i^c=\Omega\backslash A_i$ the complement of $A_i$.   Start with the inclusion-exclusion principle
\begin{equation} \label{eq:buchta_incl_excl}
\ind_{A_1\cap\ldots\cap A_l\cap A_{l+1}^c \cap \ldots \cap A_n^c}
=
\prod_{i=1}^l (1-\ind_{A_i^c})
\prod_{k=l+1}^n \ind_{A_k^c}
=
\sum_{J\subset\{1,\ldots,l\}} (-1)^{l-\#J} \ind_{\cap_{i\notin J} A_i^c},
\end{equation}
where $J=\varnothing$ is allowed in the summation and the intersection over an empty index set is $\Omega$.
Let now $A_i$ be the random event $\{X_i \text{ is a vertex of } \Pi\}$, for $i=1,\ldots,n$. Recall that $\cF_0(\Pi)$ denotes the set of vertices of $\Pi$. Then, \eqref{eq:buchta_incl_excl} becomes what can be considered as a pointwise version of Buchta's identity
\begin{equation} \label{eq:buchta_pointwise}
\ind_{\{\cF_0(\Pi) = \{X_1,\ldots,X_l\}\}}
=  \sum_{J\subset\{1,\ldots,l\}} (-1)^{l-\#J} \ind_{\{X_i, i\notin J, \text{ are not vertices of } \Pi\}}.
\end{equation}
Taking the expectation on both sides of~\eqref{eq:buchta_pointwise} yields~\eqref{eq:buchta} because by exchangeability,
$$
\P[N_n = l] = \binom nl \E \ind_{\{\cF_0(\Pi) = \{X_1,\ldots,X_l\}\}}
$$
and for every $J\subset\{1,\ldots,l\}$ with $\#J=j$,
\begin{align*}
\P [X_i, i\notin J, \text{ are not vertices of } \Pi]
&=
\P[\{X_i\colon i\notin J\} \subset\conv(X_k\colon k\in J)]\\
&
=\E [M_{j}^{n-j}].
\end{align*}
To prove the latter identity note that the conditional probability that $\{X_i\colon i\notin J\} \subset\conv(X_k\colon k\in J)$  given $\{X_k\colon k\in J\}$ equals $M_j^{n-j}$. Unlike most proofs of the present paper, the above argument is almost purely combinatorial and does not rely on topological notions like the Euler characteristic.

\section{Proof of the generalized Euler relation}\label{sec:proof_euler}

In this section we prove Theorems~\ref{theo:euler_intersection} and~\ref{theo:euler_touch}. First we need to recall some facts about the extension of the Euler characteristic to the class of ro-polyhedra which is due to Groemer~\cite{groemer}; see also~\cite[Chapter~14.4]{schneider_weil_book}.

\subsection{Euler characteristic for ro-polyhedra}
A \textit{ro-polytope} is defined as a relative interior of some polytope.
Finite unions of ro-polytopes are called \textit{ro-polyhedra}. It is known (see~\cite{groemer} or~\cite{schneider_weil_book}, page 625, Theorem 14.4.5) that there is a unique function $\chi$ (the \textit{Euler characteristic})  defined on the family of ro-polyhedra in $\R^m$ and having the following properties:
\begin{enumerate}
\item[(a)] $\chi(\varnothing) =0$.
\item [(b)] $\chi$ is additive, that is $\chi(M\cup N) = \chi(M) + \chi(N) - \chi(M\cap N)$ for all ro-polyhedra $M$ and $N$.
\item[(c)] For a non-empty polytope $P$ we have
$\chi(P) = 1$ and  $\chi(\relint  P) = (-1)^{\dim P}$.
\end{enumerate}
For the proof of the following lemma we refer to Theorem~14.4.1 in~\cite{schneider_weil_book}.
\begin{lemma}\label{lem:chi_property}
If $A_1,\ldots,A_s\subset \R^m$ are ro-polyhedra such that  $\sum_{i=1}^s a_i \ind_{A_i}=0$ for some $a_1,\ldots,a_s\in\Z$, then
$$
\sum_{i=1}^s a_i \chi(A_i)=0.
$$
\end{lemma}

\subsection{Proof of Theorem~\ref{theo:euler_intersection}}
Recall that $\cF_k(T)$ denotes the set of $k$-dimensional faces of the polytope $T$ and $\cF_{\bullet}(T)=\cup_{k=0}^{m} \cF_k(T)$ is the set of all faces of $T$. Note that $\cF_{m} (T)$ has exactly one element, namely $T$ itself.

We can represent the polytope $T$ as a disjoint union of its relatively open faces:
\begin{equation}\label{eq:polytope_union_faces}
T= \cup_{G\in \cF_\bullet(T)} \relint  G.
\end{equation}
It was observed by Nef~\cite{nef} that together with the properties of $\chi$ this immediately implies the Euler relation:
$$
1= \chi(T) = \sum_{G\in \cF_\bullet(T)} \chi(\relint  G) = \sum_{G\in \cF_\bullet(T)} (-1)^{\dim G}.
$$
The proof of Theorem~\ref{theo:euler_intersection} is more involved. A well-known corollary of the Euler relation, see~\cite[p.~627, Eq.~(14.64)]{schneider_weil_book} or~\cite[page~137]{gruenbaum_book},  states that for every face $G\in \cF_\bullet(T)$ other than $T$ itself,
\begin{equation}\label{eq:euler_generalized}
\sum_{F\in \cF_\bullet(T) \colon F\supset \relint  G} (-1)^{\dim F} = \sum_{F\in \cF_\bullet(T) \colon F\supset G} (-1)^{\dim F} =  0.
\end{equation}
From~\eqref{eq:polytope_union_faces} and~\eqref{eq:euler_generalized} we easily obtain an inclusion-exclusion relation  for the indicator function of the interior of $T$
$$
\ind_{\interior  T} = \sum_{k=0}^m \sum_{F\in \cF_{m-k}(T)} (-1)^k \ind_{F}
$$
which holds pointwise. Multiplying both sides by $\ind_L$ and replacing $k$ by $m-k$ we obtain
$$
\ind_{L\cap \interior T} = \sum_{k=0}^m \sum_{F\in \cF_{k}(T)} (-1)^{m-k} \ind_{F\cap L}.
$$
By Lemma~\ref{lem:chi_property}, this implies that
\begin{equation}\label{eq:tech1}
\chi(L\cap \interior T) = \sum_{k=0}^m \sum_{F\in \cF_{k}(T)} (-1)^{m-k} \chi(F\cap L).
\end{equation}
The left-hand side of~\eqref{eq:tech1} equals
$$
\chi(L\cap \interior  T) =
\begin{cases}
(-1)^{m-d}, &\text{if } L\cap \interior T \neq \varnothing,\\
0, &\text{if } L\cap \interior T =\varnothing
\end{cases}
$$
because $L\cap \interior T$ is either an $(m-d)$-dimensional ro-polytope (in which case its Euler characteristic equals $(-1)^{m-d}$) or empty (in which case the Euler characteristic vanishes).  As for the right-hand side of~\eqref{eq:tech1}, any term $\chi(F\cap L)$ is either zero (if $F\cap L = \varnothing$) or $1$ (if $F\cap L$ is a non-empty polytope), hence
$$
\sum_{k=0}^m \sum_{F\in \cF_{k}(T)} (-1)^{m-k} \chi(F\cap L)
=
\sum_{k=0}^m (-1)^{m-k} a_{k}.
$$
Taking everything together, we obtain the required relation.

\subsection{Proof of Theorem~\ref{theo:euler_touch}}
As in the previous proof, start with the relation
$$
\ind_{\relint  T_1} = \sum_{k=0}^m \sum_{F\in \cF_{m-k}(T_1)} (-1)^k \ind_{F}.
$$
Multiplying both sides by $\ind_{T_2}$, using $(\relint T_1) \cap T_2 =\varnothing$, and substituting $m-k$ for $k$, we infer
$$
0 = \sum_{k=0}^m \sum_{F\in \cF_{k}(T_1)} (-1)^k \ind_{F \cap T_2}.
$$
By Lemma~\ref{lem:chi_property} this implies
$$
\sum_{k=0}^m \sum_{F\in \cF_{k}(T_1)} (-1)^k \chi(F \cap T_2) = 0.
$$
Now observe that $F\cap T_2$ is a polytope which may be empty or not, hence
$$
\chi(F\cap T_2)
=
\begin{cases}
1,&\text{if } F\cap T_2 \neq \varnothing,\\
0, &\text{if } F\cap T_2 = \varnothing.
\end{cases}
$$
It follows that
$$
\sum_{k=0}^m (-1)^k  \#\{F\in \cF_{k}(T_1)\colon F\cap T_2 \neq \varnothing \} = 0,
$$
which completes the proof.

\section{Proofs of the inclusion-exclusion formulas}
\subsection{Proof of Theorem~\ref{theo:cowan}}\label{subsec:proof_cowan}
Without loss of generality we assume that $\dim \Pi=d$, otherwise we could replace $\R^d$ by the affine hull of $X_1,\ldots,X_n$.
Define a linear operator $A:\R^n \to \R^d$ by
$$
Ae_1 = X_1,\ldots, Ae_n=X_n,
$$
where $e_1,\ldots,e_n$ is the standard basis of $\R^n$.
Consider the standard $(n-1)$-dimensional simplex
$$
S:= \{(\alpha_1,\ldots,\alpha_n)\in [0,\infty)^n \colon \alpha_1+\ldots+\alpha_n =1\}\subset \R^n.
$$
The $(k-1)$-dimensional faces of $S$ have the form
$$
S_{i_1,\ldots,i_k} = \{(\alpha_1,\ldots,\alpha_n)\in S \colon \alpha_i=0 \text{ for all } i\notin \{i_1,\ldots, i_k\}\},
$$
where $1\leq i_1<\ldots<i_k \leq n$ and $k\in \{1,\ldots,n\}$.
The next lemma provides an interpretation of Cowan's formula as a statement about the number of faces of the simplex $S$ intersected by the affine subspace $A^{-1}(X)$. A somewhat related idea was used in~\cite{kabl_vys_zaporozhets_weyl_chambers}.
\begin{lemma}\label{lem:lem}
For a point $X\in\R^d$ and any $1\leq i_1<\ldots<i_k \leq n$ the following statements are equivalent:
\begin{enumerate}
\item[(i)] $X\in \conv (X_{i_1}, \ldots, X_{i_k})$
\item[(ii)] $A^{-1}(X) \cap S_{i_1,\ldots, i_k}\neq \varnothing$.
\end{enumerate}
\end{lemma}
\begin{proof}
The set $A^{-1}(X) \cap S_{i_1,\ldots, i_k}$ is non-empty if and only if there exist $(\alpha_1,\ldots,\alpha_n)\in S$ such that $\alpha_i=0$ for all $i\notin \{i_1,\ldots, i_k\}$  and
$$
A(\alpha_1e_1+\ldots +\alpha_ne_n) = X.
$$
But in view of the definition of $A$ this means that
$$
\alpha_{i_1} X_{i_1} + \ldots + \alpha_{i_k} X_{i_k} = X,
$$
which is equivalent to $X\in \conv (X_{i_1}, \ldots, X_{i_k})$.
\end{proof}

We proceed to the proof of Theorem~\ref{theo:cowan}. Fix some $X\in\R^d$.
We are going to apply Theorem~\ref{theo:euler_intersection} to the polytope $T := S$ located in the hyperplane $H\colon \alpha_1+\ldots+\alpha_n=1$ (which we identify with $\R^{n-1}$, so that $m=n-1$) and the affine subspace $L := A^{-1}(X)\cap H$.

We argue that the codimension of $L$ in the hyperplane $H$ equals $d$.  Recall that we assume that $\Pi=\conv(X_1,\ldots,X_n)$ has full dimension $d$. Hence, the vectors $X_2-X_1,\ldots, X_n-X_1$ span the whole $\R^d$. This means that for every $x\in\R^d$ there is a real solution $(v_2,\ldots,v_n)$ to $v_2(X_2-X_1)+\ldots+v_n(X_n-X_1)=x$. Defining $v_1=-(v_2+\ldots+v_n)$ and $v=(v_1,\ldots,v_n)$ we obtain a solution $v$ to $A v = x$ in the hyperplane $H_0\colon  v_1+\ldots+v_n=0$. We have $AH_0 = \R^d$, hence $AH=\R^d$ and therefore $L=A^{-1}(X)\cap H$ has codimension $d$ in $H$.

By Lemma~\ref{lem:lem}, $c_k(X)$ defined in~\eqref{eq:def_c_k} equals  $a_{k-1}$, the number of $(k-1)$-dimensional faces of $S$ which are intersected by $L$.  Hence, by Theorem~\ref{theo:euler_intersection} we obtain that
$$
\sum_{k=1}^n (-1)^{k-1} c_k(X) = \sum_{j=0}^{n-1} (-1)^{j} a_j =
\begin{cases}
(-1)^{d}, &\text{if } L\cap \relint S\neq \varnothing,\\
0, &\text{if } L\cap \relint S = \varnothing.
\end{cases}
$$

It remains to prove that $L\cap \relint S\neq \varnothing$ if and only if $X\in \relint \Pi$. Indeed, $X$ is in the relative interior of $\conv (X_1,\ldots,X_n)$ if and only if there is a strictly positive tuple $(\alpha_1,\ldots,\alpha_n)\in S$ such that $\alpha_1X_1+\ldots+\alpha_n X_n =X$; see~\cite[Theorem 1.1.14]{schneider_book}. But this is equivalent to  $(\alpha_1,\ldots,\alpha_n) \in L\cap \relint S$.


\subsection{Proof of Theorem~\ref{theo:cowan_dual}}\label{subsec:proof_dual_cowan}
Applying Theorem~\ref{theo:cowan} in its indicator functions version~\eqref{eq:cowan_indicator} to every $\Pi_I$ and interchanging the order of summation, we obtain
\begin{align*}
\sum_{\varnothing\neq I\subset\{1,\ldots,n\}} (-1)^{\# I -1}  (-1)^{\dim \Pi_I}\ind_{\relint  \Pi_I}
&=
\sum_{\varnothing\neq I\subset\{1,\ldots,n\}} (-1)^{\# I -1}  \sum_{\varnothing \neq J\subset I} (-1)^{\#J-1} \ind_{\Pi_J}\\
&=
\sum_{\varnothing\neq J\subset\{1,\ldots,n\}} (-1)^{\# J -1} \ind_{\Pi_J} \sum_{I\supset  J} (-1)^{\#I-1}\\
&=
\ind_{\Pi},
\end{align*}
where in the last step we used that
$$
\sum_{I\supset  J} (-1)^{\#I-1} =
\begin{cases}
(-1)^{n-1}, &\text{if } J=\{1,\ldots,n\},\\
0, &\text{otherwise}.
\end{cases}
$$

\subsection{Proof of Theorem~\ref{theo:cowan_affine}}
We can assume that $F$ is a proper affine subspace because the case when $F$ is a point was treated in Theorem~\ref{theo:cowan}, while the case $F=\R^n$ reduces to the identity
$$
\sum_{k=1}^n (-1)^{k-1} \binom nk = 1.
$$

Since the problem is invariant under simultaneous translations of $X_1,\ldots,X_n$ and $F$, we can assume that $F$ contains the origin.
Denote by $S:\R^d\to F^{\bot}$ the projection on $F^{\bot}$,  the orthogonal complement of $F$.
The idea is to apply Cowan's formula~\eqref{eq:cowan} to the projected points $S(X_1),\ldots,S(X_n)$.
Note that for any set $I\subset\{1,\ldots,n\}$  the convex hull of $(X_{i}\colon i\in I)$  intersects $F$ if and only if the convex hull of $(S(X_i)\colon i\in I)$ contains the origin. This is true since convex hulls are preserved under linear maps.  We obtain
\begin{align*}
\sum_{k=1}^n (-1)^{k-1} c_k(F)
&=
\sum_{k=1}^n (-1)^{k-1} \sum_{\substack{I\subset\{1,\ldots,n\}\\\#I=k}} \ind_{\{F\cap \Pi_I \neq \varnothing\}}\\
&=
\sum_{k=1}^n (-1)^{k-1} \sum_{\substack{I\subset\{1,\ldots,n\}\\\#I=k}} \ind_{\{0\in \conv(SX_i\colon i\in I)\}}\\
&=
\begin{cases}
(-1)^{d-\dim F}, &\text{if } 0\in \relint  S(\Pi),\\
0, &\text{if } 0\notin \relint  S(\Pi),
\end{cases}
\end{align*}
where the last step is by Theorem~\ref{theo:cowan} applied to the points $SX_1,\ldots,SX_n\in F^{\bot}$.  To complete the proof, note that the origin is in the relative interior of $S(\Pi) = \conv(S(X_1),\ldots, S(X_n))$ if and only if $F$ intersects the interior of $\Pi$. This is true because the projection of the interior of $\Pi$ is the relative interior of $S(\Pi)$.



\subsection{Proof of Theorem~\ref{theo:inclusion_exclusion_faces_clean}}
The idea is as follows: we will show that, essentially, $\Pi_I$ is a face of $\Pi_J$  if and only if the affine subspace spanned by $\Pi_I$ is not intersected by $\Pi_{J\backslash I}$. After that, we can apply Theorem~\ref{theo:cowan_affine}.

Denote by $\aff B$ the affine hull of a collection of points $B$.
\begin{lemma}\label{lem:clean_face}
For non-empty sets $I\subset J\subset\{1,\ldots,n\}$ the following two conditions are equivalent:
\begin{enumerate}
\item[(i)] $\Pi_I$ is a clean face of $\Pi_J$
\item[(ii)] $\aff (X_i\colon  i\in I) \cap \conv(X_k\colon k\in J\backslash I) = \varnothing$.
\end{enumerate}
\end{lemma}
\begin{proof}[Proof of (i) $\Rightarrow$ (ii)]
Let $\Pi_I$ be a clean face of $\Pi_J$. By the definition of a face, there is an affine hyperplane $H$ separating $\R^d$ into two closed half-spaces $H_+$ and $H_-$ such that $H\cap \Pi_J = \Pi_I$ and $\Pi_J\subset H_+$. Since $\Pi_I$ is clean, we have $\Pi_I\cap (X_k\colon k\in J\backslash I)=\varnothing$ and, consequently, $H\cap (X_k\colon k\in J\backslash I) = \varnothing$. It follows that $(X_k\colon k\in J\backslash I)$ is contained in the interior of $H_+$. Hence, $\conv(X_k\colon k\in J\backslash I)$ does not intersect $H$. Since $\aff (X_i\colon  i\in I)$ is a subset of $H$, we obtain (ii).
\end{proof}
\begin{proof}[Proof of $(ii)\Rightarrow (i)$]
We argue by contradiction. First we assume that $\Pi_I$ is not a face of $\Pi_J$. Then we can represent some point of $\Pi_I$ as a non-trivial convex combination of two points from $\Pi_J$ such that at least one of the points is not contained in $\Pi_I$. That is, we have
$$
\sum_{i\in I} \alpha_i X_i = \lambda \sum_{k\in J} \beta_k X_k + (1-\lambda) \sum_{k\in J} \gamma_k X_k,
$$
where $\lambda\in (0,1)$,  and $(\alpha_i)_{i\in I}$, $(\beta_k)_{k\in J}$, $(\gamma_k)_{k\in J}$ are collections of non-negative numbers summing up to $1$.
Moreover, we can assume that, say,  $\sum_{k\in J} \beta_k X_k\notin \Pi_I$. This implies that for at least one $k_0\in J\backslash I$ we have $\beta_{k_0} > 0$. Then, we can write
$$
\frac{1}{C} \sum_{i\in I} (\alpha_i - \lambda \beta_i - (1-\lambda)\gamma_i) X_i =
\sum_{k\in J\backslash I} \frac{\lambda \beta_k + (1-\lambda) \gamma_k}{C}X_k,
$$
where $C:= \sum_{l\in J\backslash I} (\lambda \beta_l + (1-\lambda) \gamma_l) > \lambda \beta_{k_0} > 0$. The left-hand side belongs to $\aff(X_i \colon i\in I)$, whereas the right-hand side belongs to $\conv(X_k\colon k\in J\backslash I)$, a contradiction. Hence, $\Pi_I$ is a face of $\Pi_J$. In fact, condition (ii) implies that $\Pi_I \cap (X_k\colon k\in J\backslash I)=\varnothing$, hence $\Pi_I$ is a clean face of $\Pi_J$.
\end{proof}

\begin{proof}[Proof of Theorem~\ref{theo:inclusion_exclusion_faces_clean}]
Without loss of generality, we assume that the affine hull of $(X_1,\ldots,X_n)$ is $\R^d$ (equivalently, $\dim \Pi=d$). Otherwise, consider this affine hull instead of $\R^d$. 
For a non-empty set $I\subset\{1,\ldots,n\}$ write
$$
S^*(I) = \{ J\subset\{1,\ldots,n\} \colon J\supset  I, \Pi_I \text{ is a clean face of } \Pi_J\}.
$$
Let $L=L_I$ be the affine hull of $\{X_i\colon i\in I\}$. Consider
\begin{align*}
N
&:= \sum_{K\subset\{1,\ldots,n\}} (-1)^{\#K} \ind_{\{\Pi_K \cap L = \varnothing\}}\\
&= \sum_{\substack{K\subset\{1,\ldots,n\}\\K\cap I = \varnothing}} (-1)^{\#K} \ind_{\{\Pi_K \cap L = \varnothing\}}\\
&= \sum_{\substack{J\subset\{1,\ldots,n\}\\J\supset  I}} (-1)^{\#J-\#I} \ind_{\{\Pi_{J\backslash I} \cap L = \varnothing\}},
\end{align*}
where in the last equality we denoted by $J$ the disjoint union of $I$ and $K$.
Applying Lemma~\ref{lem:clean_face}, we obtain
\begin{align*}
N
&= \sum_{\substack{J\subset\{1,\ldots,n\}\\J\supset  I}} (-1)^{\#I - \#J} \ind_{\{J\in S^*(I)\}}\\
&= (-1)^{\#I} \sum_{J\in S^*(I)} (-1)^{\#J}\\
&=(-1)^{\#I-1} \sum_{j=\#I}^n (-1)^{j-1} b_j^*(I).
\end{align*}
On the other hand, by Theorem~\ref{theo:cowan_affine},
\begin{align*}
N
&= \sum_{K\subset\{1,\ldots,n\}} (-1)^{\#K} - \sum_{K\subset\{1,\ldots,n\}} (-1)^{\#K} \ind_{\{\Pi_K \cap L \neq  \varnothing\}}\\
&= \sum_{K\subset\{1,\ldots,n\}} (-1)^{\#K-1} \ind_{\{\Pi_K \cap L \neq  \varnothing\}}\\
&=
\begin{cases}
(-1)^{d-\dim \Pi_I}, &\text{if } L\cap \interior  \Pi \neq \varnothing,\\
0,               &\text{if } L\cap \interior  \Pi = \varnothing.
\end{cases}
\end{align*}
Comparing these two formulas for $N$, we obtain
$$
\sum_{j=\#I}^n (-1)^{j-1} b_j^*(I) =
\begin{cases}
(-1)^{d+ \#I -1-\dim \Pi_I}, &\text{if } L\cap \interior  \Pi \neq \varnothing,\\
0,               &\text{if } L\cap \interior  \Pi = \varnothing,
\end{cases}
$$
which completes the proof because $L\cap \interior \Pi = \varnothing$ if and only if $\Pi_I$ is a face of $\Pi$.
\end{proof}

\subsection{Proof of Theorem~\ref{theo:relation_for_f_r}}\label{subsec:proof_theo_relation_for_f_r}
For a set $I\subset\{1,\ldots,n\}$ write
$$
S(I) = \{ J\subset\{1,\ldots,n\} \colon J\supset  I, \Pi_I \text{ is a face of } \Pi_J\}.
$$
The conditions we imposed on $X_1,\ldots,X_n$ imply that every $r$-dimensional face of $\Pi_J$ must have a form $\Pi_I$ for a unique $I\subset J$ with $\#I=r+1$. Using Theorem~\ref{theo:inclusion_exclusion_faces}, we get
\begin{align*}
\sum_{\varnothing\neq J\subset\{1,\ldots,n\}} (-1)^{\#J-1} f_r(\Pi_J)
&=
\sum_{\substack{I\subset\{1,\ldots,n\}\\ \#I=r+1}} \sum_{J\in S(I)} (-1)^{\#J-1}\\
&=
(-1)^d \sum_{\substack{I\subset\{1,\ldots,n\}\\ \#I=r+1}} \ind_{\{\Pi_I \text{ is a not a face of } \Pi\}} \\
&=
(-1)^d \left(\binom {n}{r+1} - f_r(\Pi)\right),
\end{align*}
where the last step holds because there are $\binom {n}{r+1}$ subsets $I$ with $\#I=r+1$.


\bibliographystyle{alpha}
\bibliography{cowan_bib}

\end{document}